\documentclass[11pt]{amsart}
\usepackage{hyperref}
\usepackage[margin=1in]{geometry}
 \pdfoutput=1 
\usepackage{amssymb}
\usepackage{epstopdf}
\usepackage{float}
\usepackage{amsmath}

\usepackage{tikz}
\usetikzlibrary{decorations.pathreplacing,calligraphy}
\usepackage{rotating}

\usepackage{tikz-cd}

\usetikzlibrary {turtle} 
\usetikzlibrary{math}
\usetikzlibrary{positioning,calc}
\usetikzlibrary{arrows,arrows.meta}
\usetikzlibrary{decorations.markings}
\usepackage[ruled, lined, linesnumbered, commentsnumbered, longend]{algorithm2e}

\newcommand\vv{\mathbf v}
\newcommand\x{\mathbf x}
\newcommand\y{\mathbf y}
\newcommand\ww{\mathbf w}

\newcommand\comment[1]{}

\def\mylist{{0}{90}{180}{270}}
\pgfmathdeclarerandomlist{myang}{\mylist}

\newtheorem{definition}{Definition}

\newtheorem{corollary}{Corollary}
\newtheorem{lemma}{Lemma}

\newtheorem{theorem}{Theorem}

\newcommand\R{{\texttt{R}}}
\renewcommand\L{{\texttt{L}}}
\renewcommand\r{{\texttt{r}}}
\renewcommand\ll{{\texttt{l}}}
\renewcommand\S{{\texttt{S}}}
\newcommand\s{{\texttt{s}}}
\newcommand\A{{\texttt{A}}}
\newcommand\B{{\texttt{B}}}

\renewcommand\v{{\texttt{v}}}
\newcommand\w{{\texttt{w}}}

            \newcommand\tileAB[5]{
  \begin{scope}[shift = {(#1+#3/2,#2+#3/2)}]
    \begin{scope}[rotate = {#4}]
    \draw[gray,line width = 2] (-#3/2,0) arc (-90:0:#3/2);
\draw[gray,line width = 2] (0,-#3/2) arc (180:90:#3/2);    
    \end{scope}
      \end{scope}
            }

            \newcommand\tileABb[5]{
  \begin{scope}[shift = {(#1+#3/2,#2+#3/2)}]
    \begin{scope}[rotate = {#4}]
    \draw[gray,line width = 2] (-#3/2,0) arc (-90:0:#3/2);
    \draw[gray,line width = 2] (0,-#3/2) arc (180:90:#3/2);
    \draw[orange, line width = 3] (1/2,#3/2)--(-1/2,-#3/2);
    \end{scope}
      \end{scope}
  }

\newcommand\twistBB[1]{
         \clip (0,0) rectangle (1,1);
          \begin{scope}[shift = {(0.5,0.5)}]
          \begin{scope}[scale = {0.86},rotate={-10}]
          \begin{scope}[shift = {(-0.5,-0.5)}]
            \tileAB{1}{0}{1}{0}{gray!50!white};
            \tileAB{0}{1}{1}{90}{gray!50!white};
            \tileAB{0}{-1}{1}{90}{gray!50!white};
            \tileAB{-1}{0}{1}{0}{gray!50!white};            
         \end{scope}
          \end{scope}
                  \end{scope}
}

\newcommand\twistBBb[1]{
         \clip (0,0) rectangle (1,1);
          \begin{scope}[shift = {(0.5,0.5)}]
          \begin{scope}[scale = {0.86},rotate={-10}]
          \begin{scope}[shift = {(-0.5,-0.5)}]
            \tileABb{1}{0}{1}{0}{gray!50!white};
            \tileABb{0}{1}{1}{90}{gray!50!white};
            \tileABb{0}{-1}{1}{90}{gray!50!white};
            \tileABb{-1}{0}{1}{0}{gray!50!white};            
         \end{scope}
          \end{scope}
                  \end{scope}
          }

\newcommand\twistBC[1]{
         \clip (0,0) rectangle (1,1);
          \begin{scope}[shift = {(0.5,0.5)}]
          \begin{scope}[scale = {0.71},rotate={-45}]
          \begin{scope}[shift = {(-0.5,-0.5)}]
            \tileAB{1}{0}{1}{0}{gray!50!white};
            \tileAB{0}{1}{1}{90}{gray!50!white};
            \tileAB{0}{-1}{1}{90}{gray!50!white};
            \tileAB{-1}{0}{1}{0}{gray!50!white};            
         \end{scope}
          \end{scope}
                  \end{scope}
}

\newcommand\twistBCb[1]{
         \clip (0,0) rectangle (1,1);
          \begin{scope}[shift = {(0.5,0.5)}]
          \begin{scope}[scale = {0.71},rotate={-45}]
          \begin{scope}[shift = {(-0.5,-0.5)}]
            \tileABb{1}{0}{1}{0}{gray!50!white};
            \tileABb{0}{1}{1}{90}{gray!50!white};
            \tileABb{0}{-1}{1}{90}{gray!50!white};
            \tileABb{-1}{0}{1}{0}{gray!50!white};            
         \end{scope}
          \end{scope}
                  \end{scope}
          }

\begin{document}

\title{On the Boundary of the Harter-Heighway dragon curve}

  \author{{
H. A. Verrill
\date{\today}
H.A.Verrill@warwick.ac.uk
    } 
    }

    \maketitle

    \begin{abstract}
      In this article we apply an L-system to prove a recurrence
      formula
      for the length of the boundary of iterands $\mathcal C_n$
      of the well known 
      Harter-Heighway dragon curve,
      a space filling curve with
      fractal boundary.
      This leads to finding formulas for related sequences of certain
      binary strings and ternary matrices.
      This proves some long standing conjectures for the
      recurrence relation for the number of terms in the boundary of
      the dragon curve, first stated in unpublished work Daykin and Tucker
      from 1975      \cite{DaykinandTucker}.
      \end{abstract}

    \section{Introduction}

    This article proves some results about two sequences, which count
    the number of components on the left and right sides of
    the iterands of the Harter-Heighway dragon.
    The sequences appear in \cite{oeis} as sequences A227036
    and A203175,
    but formulas given there
    are conjectural.  We provide proofs of the formulas, which give
    the sequences in terms of generating function,
    appearing in   \cite{Plouffe}.
    
    The Harter-Heighway dragon curve,
 also known as the Heighway dragon, or the dragon curve
 was discovered in 1967 by John Heighway and William Harter \cite{Tab},
 \cite[1.5]{Edgar}.
    It is a fascinating curve, since it is plane-filling and has a fractal
    boundary.  There are many variations on this curve.
    See e.g., \cite{Dekking}, \cite{AH2}.    
    This curve,  which  we refer to as
    $\mathcal C_\infty$,
 is given as the limit of
    a sequence of curves, $\mathcal C_n$ for
    non-negative integers $n$.
    The curve $\mathcal C_n$ is formed of
    $2^n$ equal length line segments, which we also will refer to as the
    {\it edges}, of $\mathcal C_n$,
    with a $90^\circ$ angle between each edge.
    The curves can
    be described in various ways, as follows.

    \subsection{Construction (A): paper folding}
    The Harter-Heighway curve
    was first obtained
    by repeatedly folding a strip of paper in half $n$ times,
    and then opening out all
    the folds to have angle $90^\circ$.  The first few cases
    are shown in Figure~\ref{fig:Heighwaycurve}.

    \begin{figure}[H]
\scalebox{1.09}{      
  \begin{tikzpicture}[scale={0.8}]

   \begin{scope}[shift={(-3.5,1)},scale={2}]
     \begin{scope}[rotate={-90}]
\draw [->,thick,rounded corners=0mm, line width = 0.5mm,xshift={0mm},yshift={0mm},
  turtle/distance=1cm,turtle={home,forward
}];   \end{scope}
\node at (0.5,0.125){\texttt{A}};
   \end{scope}

   \begin{scope}[shift={(-1.25,1)},scale={2}]
     \begin{scope}[rotate={-90}]
\draw [gray!30!white,rounded corners=0mm, line width = 0.5mm,xshift={0mm},yshift={0mm},
  turtle/distance=1cm,turtle={home,forward
}];
      \end{scope}
\node at (0.17,0.32){\small \texttt{A}};
\node at (0.83,0.32){\small \texttt{B}};
\node at (0.5,0.59){\texttt{+}};
      \end{scope}

   \begin{scope}[shift={(-1.25,1)},scale={1.414},rotate={-45}]
\draw [->,thick,rounded corners=0mm, line width = 0.5mm,xshift={0mm},yshift={0mm},
  turtle/distance=1cm,turtle={home,forward, right, forward
}];
      \end{scope}

    \begin{scope}[shift={(1.35,1)},scale={1.414}]
     \begin{scope}[rotate={-45}]
\draw [gray!30!white,rounded corners=0mm, line width = 0.5mm,xshift={0mm},yshift={0mm},
  turtle/distance=1cm,turtle={home,forward, right, forward
}];
     \end{scope}
     \node at (-0.12,0.4){\tiny\texttt{A}};
\node at (0.3,0.85){\tiny\texttt{B}};
\node at (-0.1,0.8){\small\texttt{+}};
\node at (.8,0.8){\small\texttt{+}};
\node at (.85,0.4){\tiny\texttt{A}};
\node at (.85,0.1){\small\texttt{-}};
\node at (1.1,0.11){\tiny\texttt{B}};
\end{scope}

    \begin{scope}[shift={(1.35,1)},scale={1}]
\draw [->,thick,rounded corners=0mm, line width = 0.5mm,xshift={0mm},yshift={0mm},
  turtle/distance=1cm,turtle={home,right=0,
forward, right, forward, right, forward, left, forward
      }];
    \end{scope}

    \begin{scope}[shift={(4,1)},scale={1}]
\draw [gray!30!white,rounded corners=0mm, line width = 0.5mm,xshift={0mm},yshift={0mm},
  turtle/distance=1cm,turtle={home,right=0,
forward, right, forward, right, forward, left, forward
      }];
      \end{scope}

    \begin{scope}[shift={(4,1)},scale={0.707},rotate={45}]
\draw [->,thick,rounded corners=0mm, line width = 0.5mm,xshift={0mm},yshift={0mm},
  turtle/distance=1cm,turtle={home,right=0,
forward, right, forward, right, forward, left, forward, right, forward, right, forward, left, forward, left, forward
      }];
    \end{scope}

    \begin{scope}[shift={(7,1)},scale={0.707},rotate={45}]
\draw [gray!30!white,thick,rounded corners=0mm, line width = 0.5mm,xshift={0mm},yshift={0mm},
  turtle/distance=1cm,turtle={home,right=0,
forward, right, forward, right, forward, left, forward, right, forward, right, forward, left, forward, left, forward
      }];
    \end{scope}

       \begin{scope}[shift={(7,1)},scale={0.5},rotate={90}]
\draw [->,thick,rounded corners=0mm, line width = 0.5mm,xshift={0mm},yshift={0mm},
  turtle/distance=1cm,turtle={home,right=0,
forward, right, forward, right, forward, left, forward, right, forward, right, forward, left, forward, left, forward, right, forward, right, forward, right, forward, left, forward, left, forward, right, forward, left, forward, left, forward
      }];
       \end{scope}

\begin{scope}[shift={(10,1)},scale={0.5},rotate={90}]
\draw [gray!30!white,thick,rounded corners=0mm, line width = 0.5mm,xshift={0mm},yshift={0mm},
  turtle/distance=1cm,turtle={home,right=0,
    forward, right, forward, right, forward, left, forward, right, forward, right, forward, left, forward, left, forward, right, forward, right, forward, right, forward, left, forward, left, forward, right, forward, left, forward, left, forward
      }];
\end{scope}

\begin{scope}[shift={(10,1)},scale={0.35355},rotate={-225}]
\draw [->,thick,rounded corners=0mm, line width = 0.5mm,xshift={0mm},yshift={0mm},
  turtle/distance=1cm,turtle={home,right=0,
forward, right, forward, right, forward, left, forward, right, forward, right, forward, left, forward, left, forward, right, forward, right, forward, right, forward, left, forward, left, forward, right, forward, left, forward, left, forward, right, forward, right, forward, right, forward, left, forward, right, forward, right, forward, left, forward, left, forward, left, forward, right, forward, right, forward, left, forward, left, forward, right, forward, left, forward, left, forward
}];
\end{scope}

\begin{scope}[shift={(13,1)},scale={0.35355},rotate={-225}]

  \draw [gray!30!white,thick,rounded corners=0mm, line width = 0.5mm,xshift={0mm},yshift={0mm},
  turtle/distance=1cm,turtle={home,right=0,
forward, right, forward, right, forward, left, forward, right, forward, right, forward, left, forward, left, forward, right, forward, right, forward, right, forward, left, forward, left, forward, right, forward, left, forward, left, forward, right, forward, right, forward, right, forward, left, forward, right, forward, right, forward, left, forward, left, forward, left, forward, right, forward, right, forward, left, forward, left, forward, right, forward, left, forward, left, forward
}];

\draw [->,rounded corners=0mm, line width = 0.5mm,xshift={0mm},yshift={0mm},
  turtle/distance=0.707cm,turtle={home,right=-45,
forward, right, forward, right, forward, left, forward, right, forward, right, forward, left, forward, left, forward, right, forward, right, forward, right, forward, left, forward, left, forward, right, forward, left, forward, left, forward, right, forward, right, forward, right, forward, left, forward, right, forward, right, forward, left, forward, left, forward, left, forward, right, forward, right, forward, left, forward, left, forward, right, forward, left, forward, left, forward, right, forward, right, forward, right, forward, left, forward, right, forward, right, forward, left, forward, left, forward, right, forward, right, forward, right, forward, left, forward, left, forward, right, forward, left, forward, left, forward, left, forward, right, forward, right, forward, left, forward, right, forward, right, forward, left, forward, left, forward, left, forward, right, forward, right, forward, left, forward, left, forward, right, forward, left, forward, left, forward
     }];

\end{scope}

  \end{tikzpicture}
  }
  \caption{Iterates $\mathcal C_0$ to $\mathcal C_6$  of Heighway's dragon curve.
  }
  \label{fig:Heighwaycurve}

  \end{figure}
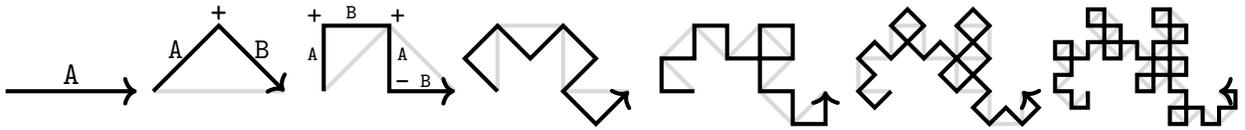

    \subsection{Construction (B): alternating triangles}
    The dragon curve
    may be described by setting $\mathcal C_0$ to be
    a unit line.
    Then $\mathcal C_n$ is obtained from
$\mathcal C_{n-1}$ as follows.
    Place right angle
    isosceles triangles on each edge of $\mathcal C_{n-1}$,
    with hypotenuse along the edge, with the
    triangles on alternating sides of the curve.  Then remove the edges of
    the original curve, as shown in 
    Figure~\ref{fig:Heighwaycurve}, where $\mathcal C_{n-1}$ is shown in gray,
    under $\mathcal C_n$ in black, so the triangles are formed with
    two black sides, and a gray hypotenuse.

    \subsection{Construction (C): L-systems}
    \label{subsec:C}
    Another method of describing the curves $\mathcal C_n$ is by the use of
    turtle geometry \cite{turtle} 
    and L-systems  \cite{LINDENMAYER1968300}.
    This is the method used in this
    article.

    An L-system $\mathcal L=(\Omega,A,P)$
    is a rewriting system, defined to be a triple, consisting of
    an alphabet of symbols, $\Omega$; a starting symbol
    $A\in\Omega^*$, called an axiom; and a function $P:\Omega\rightarrow
    \Omega^*$, from $\Omega$ to the set $\Omega^*$ of
    of finite length words with letters in $\Omega$.
    This induces a function $P:\Omega^*\rightarrow\Omega^*$;
    each letter $x$ of a word in $\Omega^*$ is replaced by $P(x)$.
    In the case of the Heighway dragon, we have
    $\Omega=\{A,B,+,-\}$, the axiom is $A$, and the function is given by
    $$P(A)=A+B, P(B)=A-B, P(+)=+, P(-)=-.$$
    The curve $\mathcal C_n$ is defined to be the path
    corresponding to the word $P^n(A)$,
    by following the letters of $P^n(A)$ as instructions for building
    a curve according to a turtle geometry construction \cite{turtle},
    with the symbols of $\Omega$ interpreted as
\begin{table}[H]
    \begin{tabular}{ll}
      $A$ & make a horizontal unit move\\
      $B$ & make a vertical unit move\\
      $+$ & turn $90^\circ$ clockwise (no move)\\
     $-$ & turn $90^\circ$ counterclockwise (no move)\\
    \end{tabular}
\end{table}
For example, labels in $\Omega$ are drawn on $\mathcal C_n$
for $n=0,1,2$ in Figure~\ref{fig:Heighwaycurve}.

   \subsection{The L-system for the boundary of
     $\mathcal C_n$.}
\label{sec:1.4}

   In \cite{Verrill_TCS}, it is shown that the boundary of the
   curve $\mathcal C_n$, or more accurately, the boundary of
   a polyomino $\mathcal S_n$ containing $\mathcal C_n$, can be described
   by an L-system.
   The polyomino $\mathcal S_n$ is defined to be the union of a collection
   of squares, each having diagonal one edge of $\mathcal C_n$,
   as shown in Figure~\ref{fig:Heighwayboundary}.

   The L-system $\mathcal L_1=(\Omega_1,Rr,P_1)$
   for the boundary of $\mathcal S_n$ is given by taking
   $\Omega=\{\R,\r,\L,\ll,\S,\s\}$, with $P_1$ given by
   \begin{equation}
     \label{eqn:L-systemforboundary}
   \R\mapsto \R\r,\rule{1cm}{0cm}
    \r\mapsto \S,\rule{1cm}{0cm}
    \L\mapsto \S,\rule{1cm}{0cm}
    \ll\mapsto \L\ll,\rule{1cm}{0cm}
    \S\mapsto \R\ll,\rule{1cm}{0cm}
        \s\mapsto \L\r.
    \end{equation}
    We consider an element $(x,y)$ of $\mathbb Z^2$ to be
    even or odd according to whether the parity of $x+y$ is even or odd.
    Each symbol of $\Omega_1$ corresponds to a path element in the turtle geometry sense,
    which consists of (1) starting from a point $(x,y)\in\mathbb Z^2$,
    a diagonal movement of length $1/\sqrt{2}$, continuing in an already specified direction, (2) no turn, or
    a  turn of $90^\circ$ left or right (3) another diagonal movement
    of length $1/\sqrt{2}$, resulting in the path which ends at one of the eight
    neighbouring points of $\mathbf Z^2$, i.e., $(x+\epsilon,y+\delta)$
    for $\epsilon,\delta\in\{-1,0,1\}$, not both zero.    
    Additional positional information gives the parity of the starting vertex.
    The symbols of $\Omega_1$ correspond to two halves of an edge of
    $\mathcal S_n$, rather than straight line edge segments in the case
    of $\mathcal C_n$ and letters in $\Omega$.
The direction of turn, or lack of turn, for each symbol, and the
starting parities, are as follows:
   \begin{table}[H]
     \begin{tabular}{ll}
       \R &  right turn in the middle of the path, starting point even\\
       \r &  right turn in the middle of the path, starting point odd\\
       \L &  left turn in the middle of the path, starting point even\\
       \ll &  left turn in the middle of the path, starting point odd\\
       \S &  no turn in the middle of the path, starting point even\\
       \s &  no turn in the middle of the path, starting point odd\\       
       \end{tabular}
   \end{table}

    Figure~\ref{fig:Heighwayboundary}
   illustrates the relationship
   between $\mathcal C_n$ and $\mathcal S_n$.  In the figure,
   in passing from $\mathcal C_n$ to $\mathcal S_{n+1}$, the unit lengths
   are drawn scaled down by a factor of $\sqrt{2}$, and the unit grid
   rotates by $45^\circ$ counterclockwise.
   The starting point is always $(0,0)$, which is even.
   The parity of the vertices of the grid through which
$\mathcal C_n$ and
   the boundary of
$\mathcal S_n$ passes are marked by black and white dots for
even and odd vertices respectively.
For example, the left boundary of $\mathcal S_4$ is described by the word
$\R\r\S\R\ll\R\r\L\ll$.

\begin{figure}[H]
  \begin{tikzpicture}[scale={0.8}]

   \begin{scope}[shift={(-4,1)},scale={2}]

     \begin{scope}[rotate={-90}]
\draw [gray!40!white,rounded corners=0mm, line width = 0.25mm,xshift={0mm},yshift={0mm},
  turtle/distance=1cm,turtle={home,forward
}];   \end{scope}
   \end{scope}

  \begin{scope}[shift={(-4,1)},scale={0.707}]

     \begin{scope}[rotate={-45}]
\draw [blue,thick,rounded corners=0mm, line width = 0.5mm,xshift={0mm},yshift={0mm},
  turtle/distance=2cm,turtle={home,forward,right, forward
}];   \end{scope}

     \begin{scope}[rotate={-135}]
\draw [red,thick,rounded corners=0mm, line width = 0.5mm,xshift={0mm},yshift={0mm},
  turtle/distance=2cm,turtle={home,forward,left, forward
}];   \end{scope}
     
     \filldraw (0,0) circle (0.15);
     \draw[fill=white] (2.75,0) circle (0.15);

     \node at (1.4,1.6){\tiny{\texttt{R}}};
          \node at (1.4,-1.6){\tiny{\texttt{L}}};
   \end{scope}

   \begin{scope}[shift={(-1.5,1)},scale={1.414},rotate={-45}]

\draw [gray!40!white,rounded corners=0mm, line width = 0.25mm,xshift={0mm},yshift={0mm},
  turtle/distance=1cm,turtle={home,forward, right, forward
}];
      \end{scope}
   \begin{scope}[shift={(-1.5,1)},scale={2}]

\draw [blue,rounded corners=0mm, line width = 0.5mm,xshift={0mm},yshift={0mm},
  turtle/distance=0.5cm,turtle={home,
    forward, right, forward, forward, right, forward
}];

\draw [red,rounded corners=0mm, line width = 0.5mm,xshift={0mm},yshift={0mm},
  turtle/distance=0.5cm,turtle={home,right=90,
    forward, forward
}];

                 \path(0,0)--++(0,0.6) node[yshift={0em}]{\tiny{\texttt{R}}}
                 --++(1,0)node[xshift={0em}]{\tiny{\texttt{r}}}
                 ;

                 \node at (0.5,-.1){\tiny{\texttt{S}}};

                 \filldraw (0,0) circle (0.05);
                      \filldraw (1,0) circle (0.05);
     \draw[fill=white] (0.5,0.5) circle (0.05);
                 
      \end{scope}

    \begin{scope}[shift={(1.35,1)},scale={1}]

\draw [gray!40!white,rounded corners=0mm, line width = 0.25mm,xshift={0mm},yshift={0mm},
  turtle/distance=1cm,turtle={home,right=0,
forward, right, forward, right, forward, left, forward
      }];
    \end{scope}

    \begin{scope}[shift={(1.35,1)},scale={1.414}]
      \node at (0.35,0.2){\tiny{\texttt{R}}};
            \node at (1.05,-0.45){\tiny{\texttt{l}}};
     \begin{scope}[rotate={45}]

\draw [blue,rounded corners=0mm, line width = 0.5mm,xshift={0mm},yshift={0mm},
  turtle/distance=0.5cm,turtle={home,
forward, right, forward, forward, right, forward, forward, forward
}];

\draw [red,rounded corners=0mm, line width = 0.5mm,xshift={0mm},yshift={0mm},
  turtle/distance=0.5cm,turtle={home,right=90,
forward, right, forward, forward, left, forward
}];

    \begin{scope}[scale={0.707},rotate={-135}]
                 \path(0,0)--++(-.5,-.5) node[yshift={-0.3em}]{\tiny{\texttt{R}}}
                 --++(-1,1)node[xshift={-0.3em}]{\tiny{\texttt{r}}}
                 --++(1,1)node[xshift={0.2em},yshift={0.3em}]{\tiny{\texttt{S}}}
                 ;

 \filldraw (0,0) circle (0.08);
                      \filldraw (-1,1) circle (0.08);
                      \draw[fill=white] (0,1) circle (0.08);
                      \draw[fill=white] (-1,0) circle (0.08);
                      \filldraw (0,2) circle (0.08);

       \end{scope}
     \end{scope}
\end{scope}

    \begin{scope}[shift={(4,1)},scale={0.707},rotate={45}]

\draw [gray!40!white,rounded corners=0mm, line width = 0.25mm,xshift={0mm},yshift={0mm},
  turtle/distance=1cm,turtle={home,right=0,
forward, right, forward, right, forward, left, forward, right, forward, right, forward, left, forward, left, forward
      }];
    \end{scope}

    \begin{scope}[shift={(4,1)},scale={1}]

\draw [blue,thick,rounded corners=0mm, line width = 0.5mm,xshift={0mm},yshift={0mm},
  turtle/distance=0.5cm,turtle={home,right=-90,
    forward, right, forward,forward, right, forward, forward, forward, forward, right, forward, forward, left,
    forward
}];

\draw [red,thick,rounded corners=0mm, line width = 0.5mm,xshift={0mm},yshift={0mm},
  turtle/distance=0.5cm,turtle={home,
    forward, right, forward,forward, right, forward, forward, left,
    forward, forward, left,
    forward
}];

    \begin{scope}[scale={0.707},rotate={-45}]
                 \path(0,0)--++(-.5,-.5) node[yshift={-0.3em}]{\tiny{\texttt{R}}}
                 --++(-1,1)node[xshift={-0.3em}]{\tiny{\texttt{r}}}
                 --++(1,1)node[xshift={-0.3em},yshift={0.3em}]{\tiny{\texttt{S}}}
                 --++(1,1)node[yshift={0.3em}]{\tiny{\texttt{R}}}
                 --++(1,-1)node[shift={(0.25em,0.25em)}]{\tiny{\texttt{l}}}
                 ;

\path(0,0)--++(-.5,.5) node[xshift={0.3em},yshift={-0.3em}]{\tiny{\texttt{R}}}
                 --++(1,1)node[xshift={-0.3em},yshift={-0.3em}]{\tiny{\texttt{r}}}
                 --++(1,-1)node[xshift={-0.3em},yshift={0em}]{\tiny{\texttt{L}}}
                 --++(1,1)node[yshift={-0.3em}]{\tiny{\texttt{r}}}
                 ;
                 
                  \filldraw (0,0) circle (0.08);
                      \filldraw (-1,1) circle (0.08);
                      \draw[fill=white] (0,1) circle (0.08);
                      \draw[fill=white] (-1,0) circle (0.08);
                      \draw[fill=white] (2,1) circle (0.08);
                      \filldraw (0,2) circle (0.08);
                      \filldraw (1,1) circle (0.08);
                      \filldraw (2,2) circle (0.08);                      
\draw[fill=white] (1,2) circle (0.08);                      
        
       \end{scope}
      \end{scope}

       \begin{scope}[shift={(7,1)},scale={0.5},rotate={90}]
\draw [gray!40!white,rounded corners=0mm, line width = 0.25mm,xshift={0mm},yshift={0mm},
  turtle/distance=1cm,turtle={home,right=0,
forward, right, forward, right, forward, left, forward, right, forward, right, forward, left, forward, left, forward, right, forward, right, forward, right, forward, left, forward, left, forward, right, forward, left, forward, left, forward
}];
       \end{scope}

       \begin{scope}[shift={(7,1)},scale={0.707}]
           \begin{scope}[rotate={135}]

\draw [blue,rounded corners=0mm, line width = 0.5mm,xshift={0mm},yshift={0mm},
  turtle/distance=0.5cm,turtle={home,right=0,
forward, right, forward, forward, right, forward, forward, forward, forward, right, forward, forward, left, forward, forward, right, forward, forward, right, forward, forward, left, forward, forward, left, forward
}];

\draw [red,rounded corners=0mm, line width = 0.5mm,xshift={0mm},yshift={0mm},
  turtle/distance=0.5cm,turtle={home,right=90,
    forward, right, forward, forward, right, forward, forward,forward, forward,
    forward, forward, left, forward, forward, left, forward
}];

           \end{scope}
               \begin{scope}[scale={0.707}]
                 \path(0,0)--++(-.5,-.5) node[yshift={-0.3em}]{\tiny{\texttt{R}}}
                 --++(-1,1)node[xshift={-0.3em}]{\tiny{\texttt{r}}}
                 --++(1,1)node[xshift={-0.3em},yshift={0.3em}]{\tiny{\texttt{S}}}
                 --++(1,1)node[yshift={0.3em}]{\tiny{\texttt{R}}}
                 --++(1,-1)node[yshift={0.5em}]{\tiny{\texttt{l}}}
                 --++(1,1)node[yshift={0.3em}]{\tiny{\texttt{R}}}
                 --++(1,-1)node[xshift={0.3em}]{\tiny{\texttt{r}}}
                 --++(-1,-1)node[xshift={0.5em}]{\tiny{\texttt{L}}}
                 --++(1,-1)node[yshift={0.5em}]{\tiny{\texttt{l}}}
                 --++(0.5,0.5)         
                 ;

  \filldraw (0,0) circle (0.1);
                      \filldraw (-1,1) circle (0.1);
                      \draw[fill=white] (0,1) circle (0.1);
                      \draw[fill=white] (-1,0) circle (0.1);
                      \draw[fill=white] (2,1) circle (0.1);
                      \filldraw (0,2) circle (0.1);
                      \filldraw (1,1) circle (0.1);
                      \filldraw (2,2) circle (0.1);
                         \filldraw (2,0) circle (0.1);
                         \filldraw (3,1) circle (0.1);
                      \filldraw (3,-1) circle (0.1);                         
                      \filldraw (4,0) circle (0.1);   
                      \draw[fill=white] (1,2) circle (0.1);
                      \draw[fill=white] (3,2) circle (0.1);  
                      \draw[fill=white] (3,0) circle (0.1);
                      \draw[fill=white] (4,-1) circle (0.1);
                      
       \end{scope}
                      \end{scope}

 \begin{scope}[shift={(10,1)},scale={0.35355},rotate={-225}]
\draw [gray!40!white,thin,rounded corners=0mm, line width = 0.25mm,xshift={0mm},yshift={0mm},
  turtle/distance=1cm,turtle={home,right=0,
forward, right, forward, right, forward, left, forward, right, forward, right, forward, left, forward, left, forward, right, forward, right, forward, right, forward, left, forward, left, forward, right, forward, left, forward, left, forward, right, forward, right, forward, right, forward, left, forward, right, forward, right, forward, left, forward, left, forward, left, forward, right, forward, right, forward, left, forward, left, forward, right, forward, left, forward, left, forward
}];
\end{scope}

\begin{scope}[shift={(10,1)},scale={0.5},rotate={180}]
\draw [blue,thick,rounded corners=0mm, line width = 0.5mm,xshift={0mm},yshift={0mm},
  turtle/distance=0.5cm,turtle={home,right=0,
forward, right, forward, forward, right, forward, forward, forward, forward, right, forward, forward, left, forward, forward, right, forward, forward, right, forward, forward, left, forward, forward, left, forward, forward, right, forward, forward, right, forward, forward, forward, forward, forward, forward, left, forward, forward, left, forward
}];

\draw [red,thick,rounded corners=0mm, line width = 0.5mm,xshift={0mm},yshift={0mm},
  turtle/distance=0.5cm,turtle={home,right=90,forward,
    right, forward, forward,right, forward, forward, forward, forward,right, forward, forward, left,  forward, forward,right,  forward, forward,left,  forward, forward, forward, forward,left,  forward, forward,left,forward
    }];

   \begin{scope}[scale={0.707},rotate={-135}]
 \filldraw (0,0) circle (0.2);
                      \filldraw (-1,1) circle (0.2);
                      \draw[fill=white] (0,1) circle (0.2);
                      \draw[fill=white] (-1,0) circle (0.2);
                      \draw[fill=white] (2,1) circle (0.2);
                      \filldraw (0,2) circle (0.2);
                      \filldraw (1,1) circle (0.2);
                      \filldraw (2,2) circle (0.2);
                         \filldraw (2,0) circle (0.2);
                         \filldraw (3,1) circle (0.2);
                      \filldraw (3,-1) circle (0.2);                         
                      \filldraw (4,0) circle (0.2);   
                      \draw[fill=white] (1,2) circle (0.2);
                      \draw[fill=white] (3,2) circle (0.2);  
                      \draw[fill=white] (3,0) circle (0.2);
                      \draw[fill=white] (4,-1) circle (0.2);

                      \draw[fill=white] (5,0) circle (0.2);
                      \draw[fill=white] (3,-2) circle (0.2);  
                      \draw[fill=white] (2,-1) circle (0.2);
                      \draw[fill=white] (2,-3) circle (0.2);
                      \draw[fill=white] (3,-4) circle (0.2);
                      \draw[fill=white] (4,-5) circle (0.2);
                      \filldraw (2,-2) circle (0.2);
                      \filldraw (2,-4) circle (0.2);
                      \filldraw (3,-5) circle (0.2);
                      \filldraw (4,-4) circle (0.2);
                      \filldraw (5,-1) circle (0.2);
                      \filldraw (4,-2) circle (0.2);
                      \filldraw (3,-1) circle (0.2);                      
   \end{scope}

\end{scope}

\begin{scope}[shift={(13.5,1)},scale={0.125},rotate={0}]

\draw [gray!40!white,thin,rounded corners=0mm, line width = 0.25mm,xshift={0mm},yshift={0mm},
  turtle/distance=1cm,turtle={home,right=90,
forward, right, forward, right, forward, left, forward, right, forward, right, forward, left, forward, left, forward, right, forward, right, forward, right, forward, left, forward, left, forward, right, forward, left, forward, left, forward, right, forward, right, forward, right, forward, left, forward, right, forward, right, forward, left, forward, left, forward, left, forward, right, forward, right, forward, left, forward, left, forward, right, forward, left, forward, left, forward, right, forward, right, forward, right, forward, left, forward, right, forward, right, forward, left, forward, left, forward, right, forward, right, forward, right, forward, left, forward, left, forward, right, forward, left, forward, left, forward, left, forward, right, forward, right, forward, left, forward, right, forward, right, forward, left, forward, left, forward, left, forward, right, forward, right, forward, left, forward, left, forward, right, forward, left, forward, left, forward, right, forward, right, forward, right, forward, left, forward, right, forward, right, forward, left, forward, left, forward, right, forward, right, forward, right, forward, left, forward, left, forward, right, forward, left, forward, left, forward, right, forward, right, forward, right, forward, left, forward, right, forward, right, forward, left, forward, left, forward, left, forward, right, forward, right, forward, left, forward, left, forward, right, forward, left, forward, left, forward, left, forward, right, forward, right, forward, left, forward, right, forward, right, forward, left, forward, left, forward, right, forward, right, forward, right, forward, left, forward, left, forward, right, forward, left, forward, left, forward, left, forward, right, forward, right, forward, left, forward, right, forward, right, forward, left, forward, left, forward, left, forward, right, forward, right, forward, left, forward, left, forward, right, forward, left, forward, left, forward, right, forward, right, forward, right, forward, left, forward, right, forward, right, forward, left, forward, left, forward, right, forward, right, forward, right, forward, left, forward, left, forward, right, forward, left, forward, left, forward, right, forward, right, forward, right, forward, left, forward, right, forward, right, forward, left, forward, left, forward, left, forward, right, forward, right, forward, left, forward, left, forward, right, forward, left, forward, left, forward, right, forward, right, forward, right, forward, left, forward, right, forward, right, forward, left, forward, left, forward, right, forward, right, forward, right, forward, left, forward, left, forward, right, forward, left, forward, left, forward, left, forward, right, forward, right, forward, left, forward, right, forward, right, forward, left, forward, left, forward, left, forward, right, forward, right, forward, left, forward, left, forward, right, forward, left, forward, left, forward, left, forward, right, forward, right, forward, left, forward, right, forward, right, forward, left, forward, left, forward, right, forward, right, forward, right, forward, left, forward, left, forward, right, forward, left, forward, left, forward, right, forward, right, forward, right, forward, left, forward, right, forward, right, forward, left, forward, left, forward, left, forward, right, forward, right, forward, left, forward, left, forward, right, forward, left, forward, left, forward, left, forward, right, forward, right, forward, left, forward, right, forward, right, forward, left, forward, left, forward, right, forward, right, forward, right, forward, left, forward, left, forward, right, forward, left, forward, left, forward, left, forward, right, forward, right, forward, left, forward, right, forward, right, forward, left, forward, left, forward, left, forward, right, forward, right, forward, left, forward, left, forward, right, forward, left, forward, left, forward
          }];

\end{scope}
    
\begin{scope}[shift={(13.5,1)},scale={0.0883},rotate={0}]
  \draw [blue,thick,rounded corners=0mm, line width = 0.5mm,xshift={0mm},yshift={0mm},
    turtle/distance=1cm,turtle={home,right=45,
  forward, right, forward, forward, right, forward, forward, forward, forward, right, forward, forward, left, forward, forward, right, forward, forward, right, forward, forward, left, forward, forward, left, forward, forward, right, forward, forward, right, forward, forward, forward, forward, forward, forward, left, forward, forward, left, forward, forward, right, forward, forward, right, forward, forward, forward, forward, right, forward, forward, left, forward, forward, right, forward, forward, left, forward, forward, forward, forward, left, forward, forward, left, forward, forward, right, forward, forward, right, forward, forward, forward, forward, right, forward, forward, left, forward, forward, right, forward, forward, right, forward, forward, left, forward, forward, left, forward, forward, right, forward, forward, right, forward, forward, left, forward, forward, left, forward, forward, right, forward, forward, left, forward, forward, forward, forward, left, forward, forward, left, forward, forward, right, forward, forward, right, forward, forward, forward, forward, right, forward, forward, left, forward, forward, right, forward, forward, right, forward, forward, left, forward, forward, left, forward, forward, right, forward, forward, right, forward, forward, forward, forward, forward, forward, left, forward, forward, left, forward, forward, right, forward, forward, right, forward, forward, forward, forward, forward, forward, left, forward, forward, left, forward, forward, right, forward, forward, right, forward, forward, left, forward, forward, left, forward, forward, right, forward, forward, left, forward, forward, forward, forward, left, forward, forward, left, forward
  }];

  \draw [red,thick,rounded corners=0mm, line width = 0.5mm,xshift={0mm},yshift={0mm},
    turtle/distance=1cm,turtle={home,right=135,
forward, right, forward, forward, right, forward, forward, forward, forward, right, forward, forward, left, forward, forward, right, forward, forward, right, forward, forward, left, forward, forward, left, forward, forward, right, forward, forward, right, forward, forward, forward, forward, forward, forward, left, forward, forward, left, forward, forward, right, forward, forward, right, forward, forward, forward, forward, right, forward, forward, left, forward, forward, right, forward, forward, left, forward, forward, forward, forward, left, forward, forward, left, forward, forward, right, forward, forward, right, forward, forward, forward, forward, right, forward, forward, left, forward, forward, right, forward, forward, left, forward, forward, forward, forward, left, forward, forward, left, forward, forward, right, forward, forward, right, forward, forward, forward, forward, forward, forward, left, forward, forward, left, forward, forward, right, forward, forward, right, forward, forward, left, forward, forward, left, forward, forward, right, forward, forward, left, forward, forward, forward, forward, left, forward, forward, left, forward
      }];
  
\end{scope}

  \end{tikzpicture}
  \caption{Iterates
    of the  boundary of $\mathcal S_n$, for $n=0$ to $5$ and $8$.
    The left boundary is in blue, the right boundary is in red, and
    the curve $\mathcal C_n$ is in gray.  The red and blue curves together
    form the boundary of the polyomino $\mathcal S_n$.
  }
  \label{fig:Heighwayboundary}

  \end{figure}
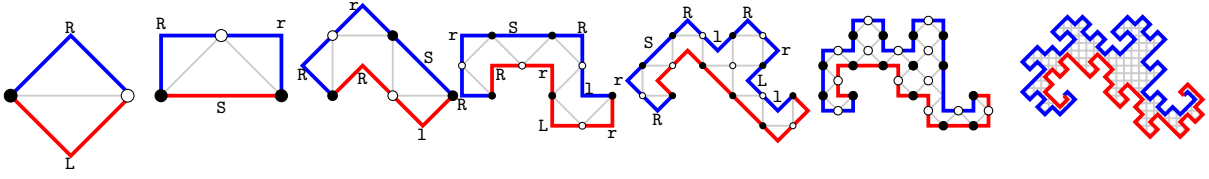

   In \cite{Verrill_TCS} it is shown that this construction produces
   a self-avoiding description of the boundary of
   $\mathcal S_n$.  Thus the number of edges of
   $\mathcal S_n$ is equal to the number of letters in
   $P^n(\R\r)$.
   We also write $||w||$ to mean the number of letters in
   $w$ for $w$ a word with letters in some alphabet.

   The L-system $\mathcal L_1$ gives
   a sequence of words $P^n(\R\r)$ describing the complete boundary
   of $\mathcal S_n$.  However, since $\mathcal C_n$ is a curve with
   different start and end points, it is considered to have two
   sides, a left and right side.
   So we also consider the curves defined by the L-systems
   $\mathcal L_L=(\Omega_1,\R,P_1)$ and
   $\mathcal L_R=(\Omega_1,\L,P_1)$, which describe the left and
   right boundaries of $\mathcal S_n$ respectively.

In this paper we consider the boundary of the curves to be the boundary of the
polyomino, $\mathcal S_n$. In the limit as $n$ tends to infinity, the
shapes $\mathcal S_n$ and $\mathcal C_n$ are the same.  However, the actual
unit length of the boundary is different for these two curves.
The difference between the two boundaries can be observed in Figure~\ref{fig:Heighwayboundary}, where the gray curve is $\mathcal C_n$, and the red curve
the boundary of $\mathcal S_n$.

\section{The left side of the Heighway dragon $\mathcal C_n$}

Given the L-system description of the boundary (\ref{eqn:L-systemforboundary}), we can prove
the following result, which has been a
long standing conjecture on the length of the boundary of the Heighway
dragon, a sequence starting
$2, 4, 8, 16, 28,\dots $, \cite[A227036]{oeis}.

\begin{theorem}
\label{thm:1}
  The length of the boundary of the Heighway dragon curve $\mathcal C_n$,
  that is, the number of horizontal and vertical line
  segments on the left side of the curve
  $\mathcal C_n$,
  is equal to
  the coefficients of the expansion of the Taylor series
  about $x=0$ of
  \begin{equation}
    \label{eqn:pre_taylor}
    \frac{2(1+x^2)}{(1-x)(1-x-2x^3)}
=\sum_{n=0}^\infty a_nx^n=2 + 4x + 8x^2 + 16x^3 + 28x^4 + 48x^5 + 84x^6 + \cdots.
  \end{equation}
\end{theorem}
\begin{proof}
  For $\w\in\Omega_1^*$ and $\texttt{X}\in\Omega_1$
  let $a_{\texttt{X}}$ denote the number of times $\texttt{X}$ occurs in
$\w$. Set
  $$v(\w):=(a_{\R},a_{\r},a_{\L},a_{\l},a_{\S}),$$
  which gives a count of the number of times each letter of
  $\Omega_1$ occurs in $\w$.
  Note that we do not include $a_{\s}$ since it is not in the
  image of any element of $\Omega_1$ under the action of $P_1$,
  where $P_1$ is the map corresponding to the L-system for the
  boundary of the polyomino containing the
  Harter-Heighway dragon curve, as in
  \ref{eqn:L-systemforboundary}.
  Define a matrix $M$,
with columns $v(P(\texttt{X}))$ for ${\texttt{X}}= \R,\r,\L,\l,\S,\s$,
so that
$$v(P(\w))=Mv(\w).$$
We have
\begin{equation}
  \label{eqn:matrixforheighway}
M=
\begin{pmatrix}
  1 & 0 & 0 & 0 & 1\\
  1 & 0 & 0 & 0 & 0\\
  0 & 0 & 0 & 1 & 0\\  
  0 & 0 & 0 & 1 & 1\\
  0 & 1 & 1 & 0 & 0
\end{pmatrix}.
\end{equation}
Note that the left side of $\mathcal S_0$ corresponds to the word
$\R$, as illustrated in Figure~\ref{fig:Heighwayboundary}.
Since $M(v(\w))$ counts the occurrences of each letter in
$P(\w)$, we have that the components of
$M^n(v(\R))$ count the number of occurrences of each
letter of $\Omega_1$.
The total number of edges of the left
boundary of $\mathcal P_n$ is thus given by the sum of these components,
i.e., since $v(\R)=(1,0,0,0,0)$, the value
$$ (1,1,1,1,1)M^n(1,0,0,0,0)^T.$$
  Examining the relationship between the curve $\mathcal C_n$ and the
  polyomino $\mathcal S_n$,
  shown in Figure~\ref{fig:Heighwayboundary},
  we see that we have
  a correspondence:
\begin{center}
  \begin{tabular}{ccc}
    segment in left side of $\mathcal S_n$ &&  $\A$ or $\B$ segment in left side boundary of $\mathcal C_n$\\
    \hline
    $\R$ or $\r$  & $\rightarrow$& $1$\\
    $\S$ or $\s$  & $\rightarrow$& $2$\\
    $\L$ or $\ll$  & $\rightarrow$& $3$.
  \end{tabular}
\end{center}
Therefore, given that our L-system gives us the
matrix $M:=M_{P_1}$ in (\ref{eqn:matrixforheighway}), we must have that
the number of $\A$ and $\B$ segments comprising the left boundary of
$\mathcal C_n$ is given by
$$b_n:=(1,1,3,3,2)M_{P_1}^n(1,0,0,0,0)^T.$$
We compute that the first few terms of this sequence are
$1, 2, 4, 8, 16, 28$.
The characteristic polynomial of
$M$ is
$$x^5 - 2x^4 + x^3 - 2x^2 + 2x=x(x-1)(x^3-x^2-2).$$
So we have
\begin{equation}
  \label{eqn:firstMatrixcomp}
\begin{array}{ll}
0&=
(1,1,3,3,2)(M^{5+n} - 2M^{4+n} + M^{3+n} - 2M^{2+n} + 2M^{n+1})(1,0,0,0,0)^T\\
&=b_{n+5} - 2b_{n+4} + b_{n+3} - 2b_{n+2} + 2b_{n+1}
\end{array}
\end{equation}
from which we obtain, for $n\ge 5$ the relationship
\begin{equation}
\label{eqn:b_n_rel}
b_n=2b_{n-1} - b_{n-2} + 2b_{n-3} - 2b_{n-4}.
\end{equation}

Now we turn to the Taylor series.
With the $a_n$ as in  (\ref{eqn:pre_taylor}), and $a_n=0$ for $n<0$,
we have 
\begin{align*}
    2(1+x^2)&=\sum_{n=0}^\infty a_n(1-x)(1-x-2x^3)x^n\\
    &=\sum_{n=0}^\infty a_n(2x^4 - 2x^3 + x^2 - 2x + 1)x^n\\
    &= \sum_{n=4}^\infty (2a_{n-4} - 2a_{n-3} + a_{n-2}-2a_{n-1}+a_n)x^n.
\end{align*}
Therefore, equating coefficients of $x^n$, we have
\begin{align*}
  2 &= a_0 (n=0) \\
  0 &= -2a_0 +a_1 (n=1)\\ 
  2 &= a_0 - 2a_1 + a_2 (n=2)\\
  0 &= 2a_{n-4} - 2a_{n-3} + a_{n-2} - 2a_{n-1} + a_n (n\ge 3).
\end{align*}
So, we have that the $a_n$ satisfy a recurrence relation,
\begin{equation}
\label{eqn:a_n_rel}
  a_n=2a_{n-1} - a_{n-2} + 2a_{n-3} - 2a_{n-4}
  \end{equation}
with the initial terms given by

\begin{align*}
  a_n &= 0 \text{ for } n<0\\
  a_0 &= 2\\
  a_1 &= 4\\
  a_2 &= 8\\
  a_3 &=16
\end{align*}
Comparing
equations
(\ref{eqn:b_n_rel})
and
(\ref{eqn:a_n_rel})
we see that the the two sequences satisfy the
same recurrence relation.
Also the first few terms are the same, up to an offset,
with $a_n=b_{n+1}$.  Therefore, the two sequences
are equal (up to the offset).
\end{proof}

\section{Right side boundary of $\mathcal S_n$ and A203175.}
\label{sec:rightside}

In this section we discuss the number of edges of
the right side boundary of $\mathcal S_n$, and find that
  this is the same as the sequence \cite[A203175]{oeis}.
  In \cite[A203175]{oeis}, this sequence has several conjectured
  descriptions, which with our L-system $(\Omega_1,\R,P_1)$
  from Section~\ref{sec:1.4}
  can now be proved.
  We will consider each description in turn, and show that
  in each case we have the same recurrence relation satisfied,
and the same first few terms, so the sequences are the same.

\subsection{Right side boundary of $\mathcal S_n$}

In the notation of Section~\ref{subsec:C},
the right side of the boundary of
the polyomino $\mathcal S_n$ is described by the word
$P_1^n(\L)$, for $n\ge 0$.
We define a sequence
\begin{equation}
  \label{eqn:def:a_n}
  a_n=|P_1^n(\L)|,
  \end{equation}
which by \cite{Verrill_TCS} counts the length of the right side of the
boundary of $\mathcal S_n$.
The first few words $P_1^n(\L)$ are
$\L, \S, \R\ll, \R\r\L\r,
\R\r\S\S\L\ll,
\R\r\S\R\ll\R\l\S\L\ll$,
corresponding to the red (lower) sides of the
polyominos in Figure~\ref{fig:Heighwayboundary}.
So the first few terms of the sequence $a_n$ are
\begin{equation}
  \label{eqn:firstfew+a_n}
1,1,2,4,6,10.
\end{equation}
\begin{theorem}
\label{thm:2}
  The length of the right side boundary of the polyomino $\mathcal S_n$,
  containing the
  Harter-Heighway dragon curve $\mathcal C_n$ is given by the sequence
  $a_n$ with
\begin{equation}
  \label{eqn:recfor_boundary_a_n}
\begin{array}{ll}
  &a_n=a_{n-1} + 2a_{n-3}\text{ for } n\ge 4\\
&
a_0=1,\>\>
a_1=1,\>\>
a_2=2.
\end{array}
\end{equation}
\end{theorem}
\begin{proof}
  As in the proof of Theorem~\ref{thm:1}, there is a matrix
  $M$, given by (\ref{eqn:matrixforheighway})
  such that $M^n(v(\L))$ counts the number of each kind of
  right boundary unit of $\mathcal S_n$.  Since $v(\L)=(0,0,1,0,0)$,
the  total number of right
  edges is
\begin{equation}
  a_n=(1,1,1,1,1)M^n(0,0,1,0,0)^T.
\end{equation}

Note that $P_1$ in
(\ref{eqn:L-systemforboundary})
is invariant under switching $\L\leftrightarrow \r$,
$\R\leftrightarrow \ll$ (in domain and image), which corresponds to
conjugating $M$ by the permutation matrix
\begin{equation}
P=
  \begin{pmatrix}
  0 & 0 & 0 & 1 & 0\\
  0 & 0 & 1 & 0 & 0\\
  0 & 1 & 0 & 0 & 0\\  
  1 & 0 & 0 & 0 & 0\\
  0 & 0 & 0 & 0 & 1
\end{pmatrix},
\end{equation}
i.e., $PMP=M$, and so we also have 
\begin{align*}
|P_1^n(\r)|&=(1,1,1,1,1)M^n(0,1,0,0,0)^T
=(1,1,1,1,1)PM^nP(0,1,0,0,0)^T\\
&=(1,1,1,1,1)M^n(0,0,1,0,0)^T=a_n=|P_1^n(\L)|.
\end{align*}
So we can write
$$a_n=|P_1^n(\L)|+|P_1^n(\r)| =
(1,1,1,1,1)M^n(0,1,1,0,0)^T.$$
Now notice that
$$M^3-M^2-2I=
\begin{pmatrix}
  -1& 0& 0& 1& 0\\ 0& -1& 1& 0& 0\\ 0& 1& -1& 0& 0\\ 1& 0& 0& -1& 0\\
  0& 0& 0& 0& 0
\end{pmatrix},
$$
and that $(0,1,1,0,0)=v(\r)+v(\L)$ is in the kernel of this matrix.
Thus
\begin{equation}
  \label{eqn:exampleofrecrel}
  (1,1,1,1,1)M^n(M^3-M^2-2I)(0,1,1,0,0)^T=0,
  \end{equation}
from which we obtain
\begin{equation}
  a_n - a_{n-1}-2a_{n-3}=0.
  \end{equation}
Together with the first few terms, given
in (\ref{eqn:firstfew+a_n}), we obtain the
recurrence relation
(\ref{eqn:recfor_boundary_a_n}).
  \end{proof}

\subsection{Generating function}

For integers $n\ge 1$,
let $b_n$ be the coefficients of the Taylor series expansion of
$x(1+x^2) / (1-x-2x^3)$, which is the generating function for these
coefficients, and is 
taken from \cite{Plouffe} and \cite[A203175]{oeis}.
By definition we have,
  \begin{equation}
    \label{eqn:pre_taylor}
    \frac{x(1+x^2)}{1-x-2x^3}
    =\sum_{n=1}^\infty b_nx^n=
x + x^2 + 2x^3 + 4x^4 + 6x^5 + 10x^6 + 18x^7 
 + \cdots.
  \end{equation}
\begin{theorem}
  \label{thm:fortheb_n}
  The coefficients $b_n$ of the Taylor series expansion of
  $x(1+x^2) / (1-x-2x^3)$ satisfy
\begin{equation}
  \label{eqn:recforTaylorb_n}
\begin{array}{ll}
  &b_n=b_{n-1} + 2b_{n-3}\text{ for } n\ge 4\\
&
b_1=1,\>\>
b_2=1,\>\>
b_3=2.
\end{array}
\end{equation}
\end{theorem}
\begin{proof}
  Multiplying both sides of (\ref{eqn:pre_taylor})
by $(1-x-2x^3)$, we have
\begin{align*}
    x+x^3&=\sum_{n=0}^\infty b_n(1-x-2x^3)x^n\\
    &=\sum_{n=0}^\infty b_n(1 - x - 2x^3)x^n\\
    &= \sum_{n=4}^\infty (b_{n} - b_{n-1} - 2b_{n-3})x^n.
\end{align*}
Therefore, by comparing coefficients of $x^n$, we see that
the $b_n$ satisfy the given recurrence relation.
The initial terms are taken from the expansion in
(\ref{eqn:pre_taylor}).
  \end{proof}

\subsection{Binary sequences without certain runs of zeros}
\label{subsec:S_n}

We now consider
binary sequences
without runs of zeros of length $1$ mod $3$,
following the
definition of Milan Janjic in the entry \cite[A203175]{oeis}.

\begin{definition}
  \label{def:S_n}
$S_n$ is the set of  binary sequences
  of length $n$ with no run of zeros of length $1$ mod $3$.
  Let $$c_n:=|S_n|.$$
\end{definition}
Note that here the length of a run of zeros
means the maximum length of any substring of
zeros, not the possible lengths of substrings of zeros,
for for example $00000\in S_5$. 
For example, $S_n$ and $c_n$ for
$1$ to $5$ are shown in the Table~\ref{tab:examplebinseq}.
\begin{table}
$$\begin{array}{|l|lllll|}
\hline
  n& 1 & 2 & 3 & 4 & 5\\
  c_n& 1 & 2 & 4 & 6 & 10\\
\hline
& 1(C) & 11(E)  & 111(E) & 1111(E) & 11111(E)\\
&   &     &     &      & 11100(B)\\
\cline{5-6}
&   &     &     & 1100(B) & 11001(D)\\
&   &     &     &      & 11000(A)\\
\cline{4-6}
S_n&   &     & 100(B) & 1001(D) & 10011(E)\\
\cline{5-6}
&   &     &     & 1000(A) & 10001(C)\\
\cline{3-6}
&   & 00(B)  & 001(D) & 0011(E) & 00111(E)\\
&   &     &     &      & 00100(B)\\
\cline{4-6}
&   &     & 000(A) & 0001(C) & 00011(E)\\
&   &     &        &         & 00000(B)\\
\hline
\end{array}
  $$
  \caption{Elements of $S_n$ for $n=1,\dots,5$,
    arranged according to the appending rules in Table~\ref{tab:binseqtypetransforms}.  Each sequence is followed by its type in brackets, according to
    Table~\ref{tab:binseqtype}.}
\label{tab:examplebinseq}
\end{table}
\begin{definition}
We define six types of sequences, as in Table~\ref{tab:binseqtype}.
\begin{table}[H]
\begin{tabular}{|ll|l|}
  \hline
  name & description & example\\
  \hline
  A&  sequences ending in a string of $0<m\equiv 0\mod 3$ zeros & $1011000$\\
B&  sequences ending in a string of $0<m\equiv 2\mod 3$ zeros& $10100000$\\
C&  sequences ending in a string of $0<m\equiv 0\mod 3$ zeros, followed by 1& $10001$ \\
D&  sequences ending in a string of $0<m\equiv 2\mod 3$ zeros, followed by 1 & $ 001$\\
E&  sequences ending in 11& $100111$ \\
  \hline
\end{tabular}
\caption{Types of elements of $S_n$.}
\label{tab:binseqtype}
\end{table}
\end{definition}

Now we define rules for building sequences of length
$n+1$ from sequences of length $n$, as
in Table~\ref{tab:binseqtypetransforms}.

\begin{table}[H]
\begin{tabular}{|ll|}
    \hline
  A & add $1$ to get a sequence of type C\\
  B & add $0$ to get a sequence of type A\\
    & or add $1$ to get a sequence of type D\\
  C & add $1$ to get a sequence of type E\\
    & or remove the last $1$ and add $00$  to get a sequence of type B\\
  D & add $1$ to get a sequence of type E\\
  E & add $1$ to get a sequence of type E\\
    & or remove the last $1$, and add $00$ to get a sequence of type B\\
  \hline
\end{tabular}
\caption{how to transform elements of $S_n$ to elements of $S_{n+1}$.}
\label{tab:binseqtypetransforms}
\end{table}

We denote the power set of $S_{n+1}$ by $\mathcal P(S_{n+1})$.
The rules in Table~\ref{tab:binseqtypetransforms}. can be used to define a function as follows, where
$\vv$ denotes a binary sequence of length
$n$, and $\ww$
denotes a binary sequence of length
$n-2$.

\begin{equation}
  \label{eqn:Sn_fun}
\begin{array}{ll}
f_n:S_n &\rightarrow \mathcal P(S_{n+1})\\
\vv &\mapsto \begin{cases}
  \{\vv1\} & \text{ if } \vv \text{ has type A (image type C)}\\  
  \{\vv0,\vv1\} & \text{ if } \vv \text{ has type B (image types A, D)}\\
  \{\ww011,\ww000\} & \text{ if } \text{ if } \vv=\ww01  \text{ has type C (image types E, B)}\\  
  \{\vv1\} & \text{ if } \vv  \text{ has type D (image type E)}\\
  \{\ww111,\ww100\} & \text{ if } \vv=\ww11 \text{ has type E (image types E, B)}\\  
  \end{cases}
\end{array}
\end{equation}

\begin{lemma}
  \label{lem:bijection}
  We have a disjoint union
  $$S_{n+1}=\bigcup_{\vv\in S_n} f_n(\vv)$$
  and
    $$|S_{n+1}|=\sum_{\vv\in S_n} |f_n(\vv)|.$$
\end{lemma}
\begin{proof}
  We must show that each element  $\x\in S_{n+1}$ is
  contained in exactly one of the sets of the form
  $f_n(\vv)$ for some $\vv\in S_n$.
  First we prove existence of some $\vv$ with $\x\in f_n(\vv)$.
  
  Suppose that
  $\x=\vv1$.  Then $\vv$ must be in $S_n$, and can have any type,
  and we have
  $\x\in f_n(\vv)$, since we
  can always add $1$ to any element of $S_n$ to
  obtain an element of $S_{n+1}$.  In other words, if $\x$ ends in a $1$,
  we can always remove it to obtain an element of $S_n$.

  Suppose that
  $\x=\vv0$.  Then either $\vv$ ends in
  a string of zeros of length $1$ or $2$ mod $3$.

  \begin{enumerate}
    \item
  In the case that $\vv$ ends in a string of $3k+2$
  zeros   (for some $k\in\mathbb Z$),
  $\vv\in S_n$ has type B, and $\x\in f_n(\vv)$.
  I.e., $\vv$ is obtained from $\x$ precisely by removing the last $0$
  from $\x$.
\item
    If $\vv$ ends in a string of $3k+1$ zeros,
  for example if $\x=100$,
  then $\vv\not\in S_n$.
  In this case, we have that $\x$ ends in at least two zeros.
  Either these are immediately preceded by a $1$ or a $0$. 
  Suppose  we have $\x=\y100$ for some word $\y$.
  Then $\x\in f_n(\y11)$, where $\y11$ has type E.
  E.g., $00100\in f_n(0011)$.
  \item
  If we are the the case where $\x=\y000$,
  then $\x\in f_n(\y01)$, where $\y01$ has type C.
  E.g., $001100000\in f_n(00110001)$.
  \end{enumerate}

  Now we must show that $\x$ is in $f_n(\vv)$ for some unique
  $\vv\in S_n$.
    \begin{enumerate}
    \item
  For the case that $\x$ ends in $1$, this is because
  we only obtain elements of $S_{n+1}$ in $f_n(\v)$
  by adding $1$ to the end of elements
  of $S_n$.
\item
  In the case that $\x$ ends in a zero, considering the definition, and
  (\ref{eqn:Sn_fun}),
  $\x$ must be in $f_n(\vv)$ for some $\vv$ of type B, C, or E,
  corresponding to $\x$ having type A, B, B respectively.
  \item
  In the case where $\x$ has type A, ending in $3k$ zeros, we can only
  obtain $\x$ from an element of $S_n$ by removing the last $0$,
  to obtain an element of type B.  So this gives a unique element
  $\vv$ with $\x\in f_n(\vv)$.
  \item 
    In the case where $\x$ has type B, and ends in at least two zeros,
    in which case
    $\vv$ is obtained uniquely from $\x$ by removing the last two zeros and replacing with $1$, resulting in a uniquely defined element
    either of type C or E.
      \end{enumerate}
Thus we obtain the stated equalities. 
\end{proof}

\begin{theorem}
\label{thm:recforc_n}
  The sequence $c_n=|S_n|$ satisfies a recurrence relation
  \begin{equation}
    \begin{array}{l}
      c_n=c_{n-1}+2c_{n-3}\text{ for }n\ge 4\\
c_1=1,\>c_2=2,\>c_3=4
    \end{array}
  \end{equation}
\end{theorem}
\begin{proof}
We can rewrite the map in (\ref{eqn:Sn_fun}) as follows:
\begin{align*}
  A &\mapsto C\\
  B &\mapsto AD\\
  C &\mapsto EB\\
  D &\mapsto E\\
  E &\mapsto EB  
\end{align*}
By Lemma~\ref{lem:bijection},
any element of $S_n$ is obtained uniquely from
some element of $S_{n-1}$.  Define a transition matrix
\begin{equation}
  M=\begin{pmatrix}
    0&1&0&0&0\\
    0&0&1&0&1\\
    1&0&0&0&0\\
    0&1&0&0&0\\
    0&0&1&1&1
  \end{pmatrix}
  \end{equation}
with columns corresponding
to $A, B, C, D, E$.
This tells us how to pass from elements of $S_n$ to elements of
$S_{n+1}$ by type.  
So we have that
the size of $S_n$ is given by
\begin{equation}
  |S_n| =
  (1,1,1,1,1)M^{n-1}(0,0,1,0,0)^T
  \end{equation}
where the vector $(0,0,1,0,0)$ corresponds to the initial string $1$
of length one.
We compute that the characteristic polynomial of $M$ is
$x^2(x^3 - x^2 - 2)$, and then, as in the
computation in (\ref{eqn:firstMatrixcomp}), we find that
\begin{equation}
  \label{eqn:c_nrecrel}
c_n=c_{c-1}+2c_{n-3}.
\end{equation}
The initial terms are as given in Table~\ref{tab:examplebinseq}.
\end{proof}

\subsection{Certain arrays with elements $0, 1, 2$.}
\label{subsec:A_n}
Sequence \cite[A203175]{oeis} has description as follows.

\begin{definition}
  \label{def:A_n}
  Let $A_n$ be the set of
  $n\times 2$ arrays, containing only elements of the set
  $\{0,1,2\}$, such that
  \begin{itemize}
    \item
      every 1 is immediately preceded by 0 to the left or above,
    \item
      no 0 is immediately preceded by a 0, either above or to the left,
    \item every 2 is immediately preceded by 0 1, in the two rows above.
  \end{itemize}
  I.e., if $m\in A_n$ has elements $m_{i,j}$, with $m_{0,0}=0$, then
  $m_{i,j}=1\Rightarrow m_{i-1,j}=0$ (and $i>0$)
  or $m_{i,j-1}=0$ (and $j>0$);
  $m_{i,j}=0\Rightarrow m_{i-1,j}\not=0$ (if $i>0$)
  and $m_{i,j-1}\not=0$ (if $j>0$);
  and   $m_{i,j}=2\Rightarrow m_{i,j-1}=1$ and $m_{i,j-2}=0$ (and $j>1$).
  Let
  \begin{equation}
    \label{eqn:def:d_n}
  d_n=|A_n|.
  \end{equation}
  \end{definition}

For example,
$A_1=\{(0,1)\}$, $A_2=\left\{\begin{pmatrix}0&1\\1&0\end{pmatrix}\right\}$
so $d_1=d_2=1$.
The elements of $A_5$ are shown in Figure~\ref{fig:tree1}, and $d_5=6$.

Just as we constructed $S_{n+1}$ from $S_n$ in the previous section,
we can construct $A_{n+1}$ from $A_n$ as follows.
We define ten different types of arrays, depending on the last two rows:
\begin{definition}
  \label{def:typeA-G}
  An $n\times 2$ array of elements $0,1,2$ is said to have
  type $A,B,C,D,E,F,G,H$ depending on the last row,
  according to the following table
    $$
  \begin{array}{ccccccccc}
    type & A       & B      & C    & D     & E    &  F     & G      & H \\
    row  & (0,1^0) & (0,1^x)& (0,2) &(1,0) & (1,2) & (2,0) & (2,1^0) & (2,1^x)
  \end{array}
  $$
\end{definition}
Here, in the second entry, $1^0$ means $1$ with a $0$ above it,
and $1^x$ means a $1$ with a $1$ or $2$ above it.
In the first column, $1$ always has a $0$ above it.
We can never have a row of the form $(1,1)$, since
this would have to be preceded by a row
$(0,0)$, and $0$ is not allowed to be next to $0$, so this
is impossible.
Similarly, $(2,2)$ would have to be preceded by $(1,1)$ so
is not allowed.
So the above list contains all the possible last rows of elements of
$A_n$ for all $n$ (not all of which will be achieved for all $n$).
We also refer to the last row as having the given type.

In the definition of $A_n$, we see that
the elements are built up in terms of the previous rows,
so all elements of $A_{n+1}$ can be obtained from
an element of $A_n$ by adding one more row which satisfies the
rules in Definition~\ref{def:A_n}.

Each type of matrix in $A_n$ can be extended to a matrix in $A_{n+1}$
by adding a row, with type
E.g., suppose a matrix $m$ in $A_n$ has type $A$,
and so ends with a row $m_n=(m_{n,1},m_{n,2})=(0,1^0)$.
Then for the next row, with new elements $(m_{n+1,1},m_{n+1,2})$,
we must have $m_{n+1,1}=1$, since $0$ in the first column can
only be followed by a $1$ below it.  Since $m_{n,2}$ has a $0$ above it,
we could have $m_{n+1,2}=2$.  Since
$m_{n+1,1}\not=0$, and $m_{n,2}\not=0$, we could have  $m_{n+1,2}=0$, but we can't
obtain $m_{n+1,2}=1$.   So type $A$ can be followed by type $D$ or
$E$.  By similar considerations, we obtain
the following table,
which shows all the possible ways of extending a
matrix of a given type in $A_n$ to a matrix of a given type in
$A_{n+1}$. We define a corresponding function $f$ on $\{A,B,C,D,E,F,G,H\}$
as in the column on the right in Table~\ref{tab:ABCDEFG}.

\begin{table}[H]
$$
  \begin{array}{l|l|l}
\text{    type of  } m_n & \text{possible type of } m_{n+1}&f\\  
\hline
A & D, E&  f(A)=\{D,E\}    \\
B & D   &  f(B)=\{D\} \\
C & D   &  f(C)=\{D\}  \\
D & A, G&  f(D)=\{A,G\}    \\
E & B, F&  f(E)=\{B,F\}    \\
F & A   &  f(F)=\{A\}      \\
G & B, C&  f(G)=\{B,C\}       \\
H & B   &  f(H)=\{B\}  
  \end{array}
  $$
  \caption{Rules for elements of $A_{n+1}$ following from elements of
    $A_n$.}
    \label{tab:ABCDEFG}
\end{table}

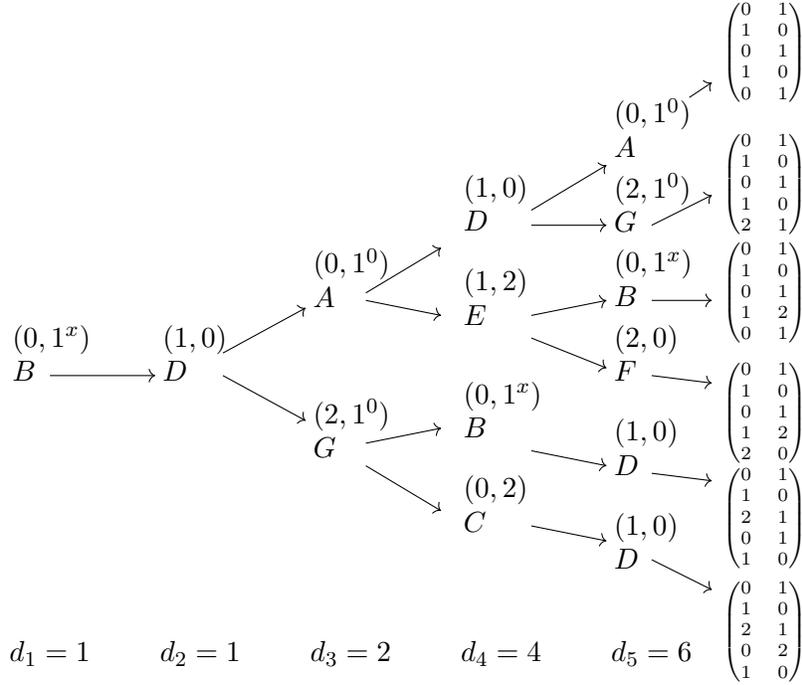
\begin{figure}
\begin{tikzpicture}
  \node at (0,0) {\parbox{1cm}{$(0,1^x)$\\ $B$}};
  \node at (2,0) {\parbox{1cm}{$(1,0)$\\ $D$}};
  \node at (4,1)  {\parbox{1cm}{$(0,1^0)$\\ $A$}};
  \node at (4,-1) {\parbox{1cm}{$(2,1^0)$\\ $G$}};
  \node at (6,2)  {\parbox{1cm}{$(1,0)$\\ $D$}};
  \node at (6,0.75) {\parbox{1cm}{$(1,2)$\\ $E$}};
  \node at (6,-0.75)  {\parbox{1cm}{$(0,1^x)$\\ $B$}};
  \node at (6,-2) {\parbox{1cm}{$(0,2)$\\ $C$}};
  \node at (8,3)  {\parbox{1cm}{$(0,1^0)$\\ $A$}};
  \node at (8,2) {\parbox{1cm}{$(2,1^0)$\\ $G$}};
  \node at (8,1)  {\parbox{1cm}{$(0,1^x)$\\ $B$}};
  \node at (8,0) {\parbox{1cm}{$(2,0)$\\ $F$}};
  \node at (8,-1.25)  {\parbox{1cm}{$(1,0)$\\ $D$}};
  \node at (8,-2.5) {\parbox{1cm}{$(1,0)$\\ $D$}};
  
  \draw[->](0,-0.3)--++(1.4,0);
  \draw[->](2.3,0)--++(1.1,0.6);
  \draw[->](2.3,-0.3)--++(1.1,-0.6);
  \draw[->](4.2,0.8)--++(1,0.6);
  \draw[->](4.2,0.7)--++(1,-0.2);
  \draw[->](4.2,-1.2)--++(1,0.2);
  \draw[->](4.2,-1.5)--++(1,-0.6);
 \draw[->](6.4,1.9)--++(1,0.6);
  \draw[->](6.4,1.7)--++(1,0);
  \draw[->](6.4,0.5)--++(1,0.2);
  \draw[->](6.4,0.2)--++(1,-0.4);
  \draw[->](6.4,-1.3)--++(1,-0.2);
  \draw[->](6.4,-2.3)--++(1,-0.2);

  \node at(9.5,4){\tiny$\begin{pmatrix}0&1\\1&0\\0&1\\1&0\\0&1\end{pmatrix}$};
  \node at(9.5,2.25){\tiny$\begin{pmatrix}0&1\\1&0\\0&1\\1&0\\2&1\end{pmatrix}$};
  \node at(9.5,0.8){\tiny$\begin{pmatrix}0&1\\1&0\\0&1\\1&2\\0&1\end{pmatrix}$};
  \node at(9.5,-0.8){\tiny$\begin{pmatrix}0&1\\1&0\\0&1\\1&2\\2&0\end{pmatrix}$};
  \node at(9.5,-2.2){\tiny$\begin{pmatrix}0&1\\1&0\\2&1\\0&1\\1&0\end{pmatrix}$};
  \node at(9.5,-3.7){\tiny$\begin{pmatrix}0&1\\1&0\\2&1\\0&2\\1&0\end{pmatrix}$};    
  \draw[->](8,-1.6)--++(0.8,-0.1);
  \draw[->](8,-0.3)--++(0.8,-0.1);
  \draw[->](8,0.7)--++(0.8,0);
  \draw[->](8,1.7)--++(0.8,0.4);
        \draw[->](8.5,3.4)--++(0.3,0.2);  
        \draw[->](8,-2.75)--++(0.8,-0.4);
        \node at (0,-4){$d_1=1$};
        \node at (2,-4){$d_2=1$};
        \node at (4,-4){$d_3=2$};
        \node at (6,-4){$d_4=4$};
        \node at (8,-4){$d_5=6$};
\end{tikzpicture}
\caption{Construction of elements of $A_5$, row by row,
  using the rules in Table~\ref{tab:ABCDEFG}}
\label{fig:tree1}
\end{figure}

\begin{theorem}
  \label{thm:forthed_n}
    The number of elements of $A_n$ in Definition~\ref{def:A_n}
    is given by
    \begin{equation}
      \begin{array}{l}
        d_n = d_{n-1}+2d_{n-3}\text{ for } n\ge 4\\
        d_1 = 1,\>
        d_2 = 1,\>
        d_3 = 2.
        \end{array}
      \end{equation}
  \end{theorem}
  \begin{proof}
    We have discussed above how rows $A, B, C, D, E, F, G, H$
    of an element of $A_n$ transition to
    the next possible row of an element of $A_{n+1}$, as shown in
    Table~\ref{tab:ABCDEFG}.
    Since row type $H$ never occurs in the image of
    $f$, we will leave this out from now on.
    The transition function $f$ in Table~\ref{tab:ABCDEFG} can be represented by the matrix
    \begin{equation}
     N= \begin{pmatrix}
        0 & 0 & 0 & 1 & 0 & 1 & 0 \\ 
        0 & 0 & 0 & 0 & 1 & 0 & 1 \\
        0 & 0 & 0 & 0 & 0 & 0 & 1 \\
        1 & 1 & 1 & 0 & 0 & 0 & 0 \\
        1 & 0 & 0 & 0 & 0 & 0 & 0 \\
        0 & 0 & 0 & 0 & 1 & 0 & 0 \\
        0 & 0 & 0 & 1 & 0 & 0 & 0 \\
        \end{pmatrix}
\end{equation}
    with rows and columns corresponding to $A$ to $G$ in alphabetical order.
    Since
     $(0,1,0,0,0,0,0)$  corresponds to $B$, in the
initial set $A_1=\{B\}$.
We have that the number of elements of each type in $A_n$ is
given by the corresponding component of the vector $N^{n-1}(0,1,0,0,0,0,0)^T$,
and the count of all of these is
\begin{equation}
  \label{eqn:firstd_n}
  d_n=(1,1,1,1,1,1,1)N^{n-1}(0,1,0,0,0,0,0)^T.
  \end{equation}
The matrix $N$ has characteristic polynomial
$x(x^3 - x^2 - 2)(x^3 + x^2 - 1)$.
The factor $(x^3 - x^2 - 2)$ corresponds to our expected recurrence relation.
However, $(0,1,0,0,0,0,0,0)$,
does not belong to the kernel of this matrix, so we can't immediately
conclude our proof.

The kernel of $N(N^3+N^2-I)$
is spanned by
$(1, 0, 0, -1, 0, 0, 0)$,
      $(0, 0, 0, 0, 1, 0, -1)$,
$(0, 1, -1, 0, 0, 0, 0)$,
      $(0, 1, 0, 0, 0, -1, 0)$,
which corresponds to a partition of
$\{A,B,\dots,G\}$ into the sets
$$ X=\{A,D\}, Y=\{E,G\}, Z=\{B,C,F\}.$$
We can rewrite Table~\ref{tab:ABCDEFG} in terms of $X, Y, Z$:
\begin{table}[H]
$$
  \begin{array}{l|l}
\text{    type of  } m_n & \text{possible type of } m_{n+1}\\  
\hline
X & X,Y\\
Y & Z,Z\\
Z & X
  \end{array}
  $$
     \caption{Rules for elements of $A_{n+1}$ following from elements of
       $A_n$, in terms of types $X, Y, Z$.}
       \label{tab:XYZ}
  \end{table}
This table shows how each type of matrix in $A_n$ can be extended to
a matrix of some type
in $A_{n+1}$, in terms of the types $X, Y, Z$.
For example, the initial element $(0,1)\in A_1$ has type $Z$, and can only be
followed by an element in $A_2$ of type $X$, which can be followed in
$A_3$ by elements of type $X$ and $Y$.
For example, to obtain the elements of $A_5$, we have sequences
corresponding to elements of $A_5$ as in Table~\ref{tab:XYZtoAB}.
\begin{table}
$$
\begin{array}{c|c|c|c|c|l|l}
  \multicolumn{5}{c|}{\text{applications of }  f} &
  \text{resulting sequence} &\text{sequence in terms of $A$ to $G$.}\\
\hline
   &  &  & X & X & ZXXXX& BDADA\\
    \cline{5-7}
    &   & X  &   & Y & ZXXXY &BDADG\\
    \cline{4-7}
    Z    &X   &   & Y & Z &ZXXYZ &BDAEB\\
        \cline{5-7}
        &   &   &  & Z &ZXXYZ &BDAEF\\
  \cline{3-7} 
  &   & Y & Z & X & ZXYZX &BDGBD\\
    \cline{4-7}
  &   &   & Z & X& ZXYZX &BDGCD\\
  \end{array}
$$
\caption{Example of words in $X, Y, Z$, and the corresponding
words in $A,B,C,D,E,F,G$.}
\label{tab:XYZtoAB}
\end{table}
We can rewrite Table~\ref{tab:XYZ} as a function
\begin{equation}
  \label{eqn:def_f}
  f(X)=\{X,Y\},\> f(Y)=\{Z,Z\}, \> f(Z)=\{X\},
  \end{equation}
where the images are not sets, but ordered lists of elements of the
set $\{X,Y,Z\}$.
The map $f$ (\ref{eqn:def_f}) can now be written in matrix format as
$$
P=\begin{pmatrix}
  1 & 0 & 1\\
  1 & 0 & 0\\
  0 & 2 & 0
\end{pmatrix},
$$
where the first, second and third
rows and columns correspond to the sets $X, Y, Z$ respectively.
Since $P$ tells us how we can continue sequences of the rows of
elements of $A_n$ to $A_{n+1}$, and the starting element 
$B$ is contained in $Z$,
which corresponds to the vector $(0,0,1)$,
heuristically, we have that
\begin{equation}
  \label{eqn:d_n_short}
  d_n=(1,1,1)P^{n-1}(0,0,1)^T.
\end{equation}
To prove more formally that (\ref{eqn:d_n_short}) holds,
we view $\mathbb R^3$ as a quotient of
$\mathbb R^7$ by the kernel of $N(N^3+N^2-I)$.  Corresponding to this
description, we find a quotient map,
$V:\mathbb R^7\rightarrow \mathbb R^3$, 
and a right inverse inclusion map $U:\mathbb R^3\hookrightarrow\mathbb R^7$.
For simplicity of notation, we  denote
the corresponding matrices by the same symbols.
The maps $U$ and $V$ are given by the matrices
\begin{equation}
  U=
  \begin{pmatrix}
    1 & 0 & 0\\
    0& 1 & 0\\
    0 & 0 & 0\\
    0 & 0 & 0\\
    0 & 0 & 1\\
    0 & 0 & 0\\
        0 & 0 & 0\\    
  \end{pmatrix},\rule{1cm}{0cm}
  V=\begin{pmatrix}
  1 & 0 & 0 & 1 & 0 & 0 & 0\\
  0 & 0 & 0 & 0 & 1 & 0 & 1\\
  0 & 1 & 1 & 0 & 0 & 1 & 0
  \end{pmatrix}.
\end{equation}
We have $P=VNU$, and can also verify that $VU=I_3$.
(the identity in $\text{GL}(n)$ is denoted by $I$ for all $n$, or $I_n$ for
clarity.)
  Define linear maps
  $Q:\mathbb R\rightarrow \mathbb R^3$,
  $R:\mathbb R\rightarrow \mathbb R^7$,
  $S:\mathbb R^3\rightarrow \mathbb R$
  and
  $T:\mathbb R^7\rightarrow \mathbb R$
  by
  \begin{equation}
    R=\begin{pmatrix}
    0\\1\\0\\0\\0\\0\\0\\0\\
    \end{pmatrix},\rule{1cm}{0cm}
        Q=\begin{pmatrix}
    0\\0\\1
        \end{pmatrix},\rule{1cm}{0cm}
    T= (1,1,1,1,1,1,1),
    \rule{1cm}{0cm}
    S=(1,1,1).
  \end{equation}
  Then the RHS of (\ref{eqn:firstd_n}) is the value of
   $TN^{n-1}R(1)$, and the RHS of
  (\ref{eqn:d_n_short}) is the value of $SP^{n-1}Q(1)$.
  So to prove that (\ref{eqn:d_n_short}) holds,
  we must check that $TN^{n-1}R=SP^{n-1}Q$ for all $n\ge 1$.
  This is equivalent to showing that the following diagram
  commutes, that is, the top row is the same map as the bottom row:
  \begin{equation}
    \label{eqn:cmmuttive}
\begin{tikzcd}[column sep=large,row sep=large]
  {\mathbb R}
   \arrow[r,"R"]
  \arrow[d,equal] & 
  {\mathbb R^7}
  \arrow[r,"N"]
  \arrow[d,swap,"V"] & 
{\mathbb R^7}
  \arrow[r,"N"]
  \arrow[d,swap,"V"]  & 
{\mathbb R^7}
\arrow[d,swap,"V"]
\arrow[r,dotted,"N^{n-4}"]
& 
{\mathbb R^7}
  \arrow[r,"N"]
  \arrow[d,swap,"V"] & 
{\mathbb R^7}
\arrow[d,swap,"V"]
\arrow[r,"T"] &
{\mathbb R}\arrow[d,equal]
  \\
   {\mathbb R}
   \arrow[r,swap,"Q"]
   &
{\mathbb R^3}
  \arrow[r,swap,"P"] & 
{\mathbb R^3}
  \arrow[r,swap,"P"] & 
        {\mathbb R^3}
        \arrow[r,swap,dotted,"P^{n-4}"]
& 
{\mathbb R^3}
  \arrow[r,swap,"P"] & 
        {\mathbb R^3}
        \arrow[r,swap,"S"]
        &
        {\mathbb R}
\end{tikzcd}
    \end{equation}
To see this, first note that we have equalities
  \begin{equation}
    Q=VR,\rule{1cm}{0cm}
    T=SV,\rule{1cm}{0cm}
    PV = VN,
    \end{equation}
  which can be checked computationally.  Then by an inductive
  argument,
  we have for $i=0,\dots,n-1$ that
  $TN^{n-1}R=SVN^{n-1}R = SP^iVN^{n-1-i}R = SP^{n-1}VR = SP^{n-1}Q$,
which just corresponds to following through the diagram,
and we see that (\ref{eqn:cmmuttive}) does commute.  So
  (\ref{eqn:d_n_short}) does indeed hold.  
  We find that the matrix $P$ has characteristic polynomial $x^3-x^2-2$,
  so $P^3-P^2-2I=0$.
  Thus  from (\ref{eqn:d_n_short}), using the same method
  as in (\ref{eqn:exampleofrecrel}), we obtain the recurrence 
$$d_n-d_{n-1}-2d_{n-3}=0, \text{ for }n\ge 3$$
with initial terms $d_1=d_2=1$ and $d_3=2$ as in Figure~\ref{fig:tree1}.
    \end{proof}

\begin{corollary}
  We have $a_n=b_{n-1}=c_n=d_{n-1}$,
  with $a_n, b_n, c_n, d_n$ as defined in
  (\ref{eqn:def:a_n}),
  (\ref{eqn:pre_taylor}),
  Definition~\ref{def:S_n},
and    (\ref{eqn:def:d_n}) 
  respectively.
\end{corollary}
\begin{proof}
  The first few terms, $1,1,2,4,6$, of these sequences are
  given in (\ref{eqn:firstfew+a_n}),
  (\ref{eqn:recforTaylorb_n}),
Table~\ref{tab:examplebinseq}, and
Figure~\ref{fig:tree1}
  respectively.  The only difference is that
  $a_n$ starts from $n=0$,
  $b_n$ starts from $n=1$,
  $c_n$ starts from $n=0$, (inserting an extra $c_0=1$ term)
  and
  $d_n$ starts from $n=1$.
  We showed that
  they all satisfy the same recurrence relation, in
  Theorems~\ref{thm:2},
  \ref{thm:fortheb_n},
  \ref{thm:recforc_n},
  and
  \ref{thm:forthed_n} respectively.
Thus the result follows.  
  \end{proof}

\section{Conclusions}

By using the L-system for the boundary of the Harter-Heighway dragon curve,
we have been able to prove results not only about the dragon curve,
but also related sequences,
found in \cite{oeis},
by also viewing them in terms of transition matrices inspired by
L-systems.
Given that the sequences in Section~\ref{sec:rightside}
count sizes of various sets, which all turn out to have the same size,
we have actually constucted bijections between these sets.
The L-system for the Heighway dragon results in a word, with letters in 
an order which is lost by just using the matrix $M$, which only counts
the total number of letters.  The elements of the sets in
$S_n$ and $A_n$ in
Sections~\ref{subsec:S_n}
and \ref{subsec:A_n} do not a proiri have a natural ordering, but we can use the
L-system for the right side of the dragon curve to impose an ordering on the elements of the sets $S_n$ and $A_n$ (Definitions~\ref{def:A_n} and \ref{def:S_n}), though we must make a choice, e.g.,
lexicographical on the elements $A,\dots,G$ in Definition~\ref{def:typeA-G}.
This gives orderings for example as in Table~\ref{tab:examplebinseq} and
Figure~\ref{fig:tree1}.
Given such a choice, this results in a corresponding bijection between the
elements of the sets $A_n$ and $S_n$ and the edges of the right side of the
dragon curve $\mathcal S_n$.
This may well just be numerology, but perhaps there is an
 interesing geometrical or
number theoretical meaning waiting to be discovered.

\bibliographystyle{plain}

\end{document}